\documentclass[a4paper,12pt]{article}

\usepackage{amsmath}
\usepackage{ntheorem}
\usepackage{amssymb}
\usepackage{latexsym}
\usepackage{url,mathrsfs}
\usepackage[all]{xy}\UseComputerModernTips
\usepackage{graphicx}

\theoremstyle{plain}
\newtheorem{proposition}{Proposition}[section]
\newtheorem{theorem}[proposition]{Theorem}
\newtheorem{lemma}[proposition]{Lemma}
\newtheorem{corollary}[proposition]{Corollary}

\theoremstyle{plain}
\theorembodyfont{\normalfont\rm}
\newtheorem{definition}[proposition]{Definition}
\newtheorem{remark}[proposition]{Remark}

\theoremstyle{nonumberplain}
\theoremheaderfont{\normalfont\itshape}
\theoremseparator{.}
\theorembodyfont{\normalfont}
\newtheorem{proof}{Proof}

\newcommand{\qed}{\hfill $\Box$}

\makeatletter

\@addtoreset{equation}{section}
\makeatother

\newcommand{\ZZ}{{\mathbb Z}}

\newcommand{\RR}{{\mathbb R}}
\newcommand{\CC}{{\mathbb C}}

\newcommand{\QQ}{{\mathbb Q}}
\newcommand{\LL}{{\mathbb L}}

\renewcommand{\d}{{\rm dim}}

\newcommand{\e}{\varepsilon}
\renewcommand{\SS}{{\mathcal S}}
\newcommand{\Spec}{{\rm Spec}}
\newcommand{\Con}{{\rm Con}}

\newcommand{\id}{{\rm id}}
\newcommand{\supp}{{\rm supp}}

\newcommand{\Int}{{\rm Int}}

\newcommand{\Var}{{\rm Var}}
\newcommand{\HSm}{{\rm HS}^{\rm mon}}

\newcommand{\relint}{{\rm rel.int}}

\newcommand{\dist}{{\rm dist}}
\newcommand{\height}{{\rm ht}}

\newcommand{\dd}{b}

\newcommand{\KK}{{\rm K}}

\newcommand{\M}{{\mathcal M}}

\newcommand{\tl}[1]{\widetilde{#1}}

\newcommand{\simto}{\overset{\sim}{\longrightarrow}}

\newcommand{\dsum}{\displaystyle \sum}

\renewcommand{\(}{\left(}
\renewcommand{\)}{\right)}

\begin{document}

\title{Motivic Milnor fibers and Jordan normal forms of Milnor monodromies
\footnote{ {\bf 2010 Mathematics Subject Classification:}14E18, 14M25, 32C38, 32S35, 32S40, 
}}

\author{Yutaka \textsc{Matsui}\footnote{Department of Mathematics, Kinki University, 3-4-1, Kowakae, Higashi-Osaka, Osaka, 577-8502, Japan. E-mail: matsui@math.kindai.ac.jp} \and Kiyoshi \textsc{Takeuchi}\footnote{Institute of Mathematics, University  of Tsukuba, 1-1-1, Tennodai, Tsukuba, Ibaraki, 305-8571, Japan. E-mail: takemicro@nifty.com}}

\date{}

\sloppy

\maketitle

\begin{abstract}
By calculating the equivariant mixed Hodge numbers of motivic Milnor fibers introduced by Denef-Loeser, we obtain explicit formulas for the Jordan normal forms of Milnor monodromies. The numbers of the Jordan blocks will be described by the Newton polyhedron of the polynomial. 
\end{abstract}

\section{Introduction}\label{sec:1}

In this paper, by using motivic Milnor fibers introduced by Denef-Loeser \cite{D-L-1} and \cite{D-L-2}, we obtain explicit formulas for the Jordan normal forms of Milnor monodromies. Let $f(x)= \sum_{v \in \ZZ_+^n} a_v x^v \in \CC[x_1,\ldots,x_n]$ be a polynomial on $\CC^n$ such that the hypersurface $f^{-1}(0)= \{ x \in \CC^n \ |\ f(x)=0 \}$ has an isolated singular point at $0\in \CC^n$. Then by a fundamental theorem of Milnor \cite{Milnor}, the Milnor fiber $F_0$ of $f$ at $0 \in \CC^n$ has the homotopy type of bouquet of $(n-1)$-spheres. In particular, we have $H^j(F_0;\CC) \simeq 0$ ($j\neq 0, \ n-1$). Denote by
\begin{equation}
\Phi_{n-1,0} \colon H^{n-1}(F_0;\CC) \simto H^{n-1}(F_0;\CC)
\end{equation}
\noindent the $(n-1)$-th Milnor monodromy of $f$ at $0 \in \CC^n$. By the theory of monodromy zeta functions due to A'Campo \cite{A'Campo} and Varchenko \cite{Varchenko} etc., the eigenvalues of $\Phi_{n-1,0}$ were fairly well-understood. See Oka's book \cite{Oka} for an excellent exposition of this very important result. However to the best of our knowledge, it seems that the Jordan normal form of $\Phi_{n-1,0}$ is not fully understood yet. In this paper, we give a combinatorial description of the Jordan normal form of $\Phi_{n-1,0}$ by using motivic Milnor fibers (For a computer algorithm by Brieskorn lattices, see Schulze \cite{Schulze} etc.).

From now on, let us assume also that $f$ is convenient and non-degenerate at $0 \in \CC^n$ (see Definitions \ref{CVN} and \ref{NDC}). Note that the second condition is satisfied by generic polynomials $f$. Then we can describe the Jordan normal form of $\Phi_{n-1,0}$ very explicitly as follows. We call the convex hull of $\bigcup_{v \in \supp (f)} \{v + \RR_+^n\}$ in $\RR_+^n$ the Newton polyhedron of $f$ and denote it by $\Gamma_+(f)$. Let $q_1,\ldots,q_l$ (resp. $\gamma_1,\ldots, \gamma_{l^{\prime}}$) be the $0$-dimensional (resp. $1$-dimensional) faces of $\Gamma_{+}(f)$ such that $q_i\in \Int (\RR_+^n)$ (resp. the relative interior $\relint(\gamma_i)$ of $\gamma_i$ is contained in $\Int(\RR_+^n)$). For each $q_i$ (resp. $\gamma_i$), denote by $d_i >0$ (resp. $e_i>0$) its lattice distance $\dist(q_i, 0)$ (resp. $\dist(\gamma_i,0)$) from the origin $0\in \RR^n$. For $1\leq i \leq l^{\prime}$, let $\Delta_i$ be the convex hull of $\{0\}\sqcup \gamma_i$ in $\RR^n$. Then for $\lambda \in \CC \setminus \{1\}$ and $1 \leq i \leq l^{\prime}$ such that $\lambda^{e_i}=1$ we set
\begin{equation}
n(\lambda)_i
= \sharp\{ v\in \ZZ^n \cap \relint(\Delta_i) \ |\ \height (v, \gamma_i)=k\} +\sharp \{ v\in \ZZ^n \cap \relint(\Delta_i) \ |\ \height (v, \gamma_i)=e_i-k\},
\end{equation}
where $k$ is the minimal positive integer satisfying $\lambda=\zeta_{e_i}^{k}$ ($\zeta_{e_i}:=\exp (2\pi\sqrt{-1}/e_i)$) and for $v\in \ZZ^n \cap \relint(\Delta_i)$ we denote by $\height (v, \gamma_i)$ the lattice height of $v$ from the base $\gamma_i$ of $\Delta_i$. Then in Section \ref{sec:4} we prove the following result which describes the number of Jordan blocks for each fixed eigenvalue $\lambda \neq 1$ in $\Phi_{n-1, 0}$. Recall that by the monodromy theorem the sizes of such Jordan blocks are bounded by $n$.

\begin{theorem}\label{thm:1-1}
Assume that $f$ is convenient and non-degenerate at $0 \in \CC^n$. Then for any $\lambda \in \CC^* \setminus \{1\}$ we have
\begin{enumerate}
\item The number of the Jordan blocks for the eigenvalue $\lambda$ with the maximal possible size $n$ in $\Phi_{n-1,0} \colon H^{n-1}(F_0;\CC) \simto H^{n-1}(F_0;\CC)$ is equal to $\sharp \{q_i \ |\ \lambda^{d_i}=1\}$.
\item The number of the Jordan blocks for the eigenvalue $\lambda$ with size $n-1$ in $\Phi_{n-1, 0}$ is equal to $\sum_{i \colon \lambda^{e_i}=1} n(\lambda)_i$.
\end{enumerate}
\end{theorem}

\noindent Namely the Jordan blocks for the eigenvalues $\lambda \neq 1$ in the monodromy $\Phi_{n-1, 0}$ are determined by the lattice distances of the faces of $\Gamma_{+}(f)$ from the origin $0 \in \RR^n$. The monodromy theorem asserts also that the sizes of the Jordan blocks for the eigenvalue $1$ in $\Phi_{n-1, 0}$ are bounded by $n-1$. In this case, we have the following result. Denote by $\Pi_f$ the number of the lattice points on the $1$-skeleton of $\partial \Gamma_{+}(f) \cap \Int (\RR^n_+)$. For a compact face $\gamma \prec \Gamma_{+}(f)$, denote by $l^*(\gamma)$ the number of the lattice points on the relative interior $\relint(\gamma)$ of $\gamma$.

\begin{theorem}\label{thm:1-2}
In the situation of Theorem \ref{thm:1-1} 
we have 
\begin{enumerate}
\item {\rm (van Doorn-Steenbrink \cite{D-St})} The number of the Jordan blocks for the eigenvalue $1$ with the maximal possible size $n-1$ in $\Phi_{n-1, 0}$ is $\Pi_f$.
\item The number of the Jordan blocks for the eigenvalue $1$ with size $n-2$ in $\Phi_{n-1, 0}$ is equal to $2 \sum_{\gamma} l^*(\gamma)$, where $\gamma$ ranges through the compact faces of $\Gamma_{+}(f)$ such that $\d \gamma =2$ and $\relint(\gamma) \subset \Int (\RR^n_+)$. In particular, this number is even.
\end{enumerate}
\end{theorem}

\noindent Note that Theorem \ref{thm:1-2} (i) was previously obtained in van Doorn-Steenbrink \cite{D-St} by using different methods. Roughly speaking, the nilpotent part for the eigenvalue $1$ in the monodromy $\Phi_{n-1, 0}$ is determined by the convexity of the hypersurface $\partial \Gamma_{+}(f) \cap \Int (\RR^n_+)$. Thus Theorems \ref{thm:1-1} and \ref{thm:1-2} generalize the well-known fact that the monodromies of quasi-homogeneous polynomials are semisimple. In fact, by our results in Sections \ref{sec:2} and \ref{sec:4} a general algorithm for computing all the spectral pairs of the Milnor fiber $F_0$ is obtained. This in particular implies that we can compute the Jordan normal form of $\Phi_{n-1, 0}$ completely. Note that the spectrum of $F_0$ obtained in Saito \cite{Saito-3} and Varchenko-Khovanskii \cite{K-V} is not enough to deduce the Jordan normal form. Moreover, if any compact face of $\Gamma_{+}(f)$ is prime (see Definition \ref{dfn:2-16}) we obtain also a closed formula for the Jordan normal form. See Section \ref{sec:4} for the details. 

This paper is organized as follows. In Section \ref{sec:2}, we introduce some generalizations of the results of Danilov-Khovanskii \cite{D-K} obtained in \cite{M-T-3}. By them we obtain a general algorithm for computing the equivariant mixed Hodge numbers of non-degenerate toric hypersurfaces. In Section \ref{sec:3}, we recall some basic definitions and results on motivic Milnor fibers introduced by Denef-Loeser \cite{D-L-1} and \cite{D-L-2}. Then in Section \ref{sec:4}, by rewriting them in terms of the Newton polyhedron $\Gamma_{+}(f)$ with the help of the results in Section \ref{sec:2} and \cite{M-T-3}, we prove various combinatorial formulas for the Jordan normal form of the Milnor monodromy $\Phi_{n-1, 0}$. Although our proof for the eigenvalue $1$ in this paper is very different from the one in \cite{M-T-3}, our results in Section \ref{sec:4} are completely parallel to those for monodromies at infinity obtained in \cite{M-T-3}. We thus find a striking symmetry between local and global. Finally, let us mention that in \cite{E-T} the results for the other eigenvalues $\lambda \not= 1$ in this paper were already generalized to the monodromies over complete intersection subvarieties in $\CC^n$.

\section{Preliminary notions and results}\label{sec:2}

In this section, we recall our results in \cite[Section 2]{M-T-3} which will be used in this paper. They are slight generalizations of the results in Danilov-Khovanskii \cite{D-K}. 

\begin{definition}\label{dfn:2-6-2}
Let $g(x)=\sum_{v \in \ZZ^n} a_vx^v$ ($a_v\in \CC$) be a Laurent polynomial on $(\CC^*)^n$.
\begin{enumerate}
\item We call the convex hull of $\supp(g):=\{v\in \ZZ^n \ |\ a_v\neq 0\} \subset \ZZ^n $ in $\RR^n$ the Newton polytope of $g$ and denote it by $NP(g)$.
\item For $u\in (\RR^n)^*$, we set $\Gamma(g;u):=\left\{ v\in NP(g)\ \left| \ \langle u,v\rangle =\min_{w\in NP(g)} \langle u,w\rangle \right.\right\}$.
\item For $u \in (\RR^n)^*$, we define the $u$-part of $g$ by $g^u(x):=\sum_{v \in \Gamma(g;u)} a_vx^v$.
\end{enumerate}
\end{definition}

\begin{definition}[\cite{Kushnirenko}]
Let $g$ be a Laurent polynomial on $(\CC^*)^n$. Then we say that the hypersurface $Z^*=\{ x\in (\CC^*)^n \ |\ g(x)=0 \}$ of $(\CC^*)^n$ is non-degenerate if for any $u \in (\RR^n)^*$ the hypersurface $\{ x\in (\CC^*)^n \ |\ g^u(x)=0 \}$ is smooth and reduced.
\end{definition}

In the sequel, let us fix an element $\tau =(\tau_1,\ldots, \tau_n) \in T:=(\CC^*)^n$ and let $g$ be a Laurent polynomial on $(\CC^*)^n$ such that $Z^*=\{ x\in (\CC^*)^n \ |\ g(x)=0\}$ is non-degenerate and invariant by the automorphism $l_{\tau} \colon (\CC^*)^n \underset{\tau \times}{\simto}(\CC^*)^n$ induced by the multiplication by $\tau$. Set $\Delta =NP(g)$ and for simplicity assume that $\d \Delta=n$. Then there exists $\beta \in \CC$ such that $l_{\tau}^*g= g \circ l_{\tau}=\beta g$. This implies that for any vertex $v$ of $\Delta =NP(g)$ we have ${\tau}^v={\tau}_1^{v_1} \cdots {\tau}_n^{v_n}=\beta$. Moreover by the condition $\d \Delta=n$ we see that $\tau_1, \tau_2, \ldots , \tau_n$ are roots of unity. For $p,q \geq 0$ and $k \geq 0$, let $h^{p,q}(H_c^k(Z^*;\CC))$ be the mixed Hodge number of $H_c^k(Z^*;\CC)$ and set
\begin{equation}
e^{p,q}(Z^*)=\dsum_k (-1)^k h^{p,q}(H_c^k(Z^*;\CC))
\end{equation}
as in \cite{D-K}. The above automorphism of $(\CC^*)^n$ induces a morphism of mixed Hodge structures $l_{\tau}^* \colon H_c^k(Z^*;\CC) \simto H_c^k(Z^*;\CC)$ and hence $\CC$-linear automorphisms of the $(p,q)$-parts $H_c^k(Z^*;\CC)^{p,q}$ of $H_c^k(Z^*;\CC)$. For $\alpha \in \CC$, let $h^{p,q}(H_c^k(Z^*;\CC))_{\alpha}$ be the dimension of the $\alpha$-eigenspace $H_c^k(Z^*;\CC)_{\alpha}^{p,q}$ of this automorphism of $H_c^k(Z^*;\CC)^{p,q}$ and set
\begin{equation}
e^{p,q}(Z^*)_{\alpha}=\dsum_k (-1)^k h^{p,q}(H_c^k(Z^*;\CC))_{\alpha}.
\end{equation}
We call $e^{p,q}(Z^*)_{\alpha}$ the equivariant mixed Hodge numbers of $Z^*$. Since we have $l_{\tau}^r =\id_{Z^*}$ for some $r \gg 0$, these numbers are zero unless $\alpha$ is a root of unity. Obviously we have
\begin{equation}
e^{p,q}(Z^*)=\dsum_{\alpha \in \CC} e^{p,q}(Z^*)_{\alpha}, \qquad 
e^{p,q}(Z^*)_{\alpha}=e^{q,p}(Z^*)_{\overline{\alpha}}.
\end{equation}
In this setting, along the lines of Danilov-Khovanskii \cite{D-K} we can give an algorithm for computing these numbers $e^{p,q}(Z^*)_{\alpha}$ as follows. First of all, as in \cite[Section 3]{D-K} we have the following result.

\begin{proposition}\label{prp:2-15} 
{\bf (\cite[Proposition 2.6]{M-T-3})} For $p,q \geq 0$ such that $p+q >n-1$, we have
\begin{equation}
e^{p,q}(Z^*)_{\alpha}=
\begin{cases}
(-1)^{n+p+1}\binom{n}{p+1} & (\text{$\alpha=1$ and $p=q$}),\\
\hspace*{10mm}0 & (\text{otherwise}),
\end{cases}
\end{equation}
(we used the convention $\binom{a}{b}=0$ ($0 \leq a <b$) for binomial coefficients).
\end{proposition}

For a vertex $w$ of $\Delta$, consider the translated polytope $\Delta^w:=\Delta -w$ such that $0 \prec \Delta^w$ and ${\tau}^v=1$ for any vertex $v$ of $\Delta^w$. Then for $\alpha \in \CC$ and $k \geq 0$ set
\begin{equation}
l^*(k\Delta)_{\alpha}=\sharp \{ v \in \Int (k\Delta^w) \cap \ZZ^n \ |\ {\tau}^v =\alpha\} \in \ZZ_+:=\ZZ_{\geq 0}.
\end{equation}
We can easily see that these numbers $l^*(k\Delta)_{\alpha}$ do not depend on the choice of the vertex $w$ of $\Delta$. We define a formal power series $P_{\alpha}(\Delta;t)=\sum_{i \geq 0}\varphi_{\alpha, i}(\Delta)t^i$ by
\begin{equation}
P_{\alpha}(\Delta;t)=(1-t)^{n+1} \left\{ \dsum_{k \geq 0} l^*(k\Delta)_{\alpha}t^k\right\}.
\end{equation}
Then we can easily show that $P_{\alpha}(\Delta;t)$ is actually a polynomial as in \cite[Section 4.4]{D-K}. 

\begin{theorem}\label{thm:2-14}
{\bf (\cite[Theorem 2.7]{M-T-3})} In the situation as above, we have
\begin{equation}
\dsum_q e^{p,q}(Z^*)_{\alpha}
=\begin{cases}
(-1)^{p+n+1}\binom{n}{p+1} +(-1)^{n+1} \varphi_{\alpha, n-p}(\Delta) & (\alpha=1), \\
(-1)^{n+1} \varphi_{\alpha, n-p}(\Delta) & (\alpha \neq 1).
\end{cases}
\end{equation}
\end{theorem}

By Proposition \ref{prp:2-15} and Theorem \ref{thm:2-14} we can now calculate the numbers $e^{p,q}(Z^*)_{\alpha}$ on the non-degenerate hypersurface $Z^* \subset (\CC^*)^n$ for any $\alpha \in \CC$ as in \cite[Section 5.2]{D-K}. Indeed for a projective toric compactification $X$ of $(\CC^*)^n$ such that the closure $\overline{Z^*}$ of $Z^*$ in $X$ is smooth, the variety $\overline{Z^*}$ is smooth projective and hence there exists a perfect pairing
\begin{equation}
H^{p,q}(\overline{Z^*};\CC)_{\alpha} \times H^{n-1-p, n-1-q}(\overline{Z^*};\CC)_{\alpha^{-1}} \longrightarrow \CC
\end{equation}
for any $p,q \geq 0$ and $\alpha \in \CC^*$ (see for example \cite[Section 5.3.2]{Voisin}). Therefore, we obtain equalities $e^{p,q}(\overline{Z^*})_{\alpha}=e^{n-1-p,n-1-q}(\overline{Z^*})_{\alpha^{-1}}$ which are necessary to proceed the algorithm in \cite[Section 5.2]{D-K}. We have also the following analogue of \cite[Proposition 5.8]{D-K}.

\begin{proposition}\label{prp:new}
{\bf (\cite[Proposition 2.8]{M-T-3})} For any $\alpha \in \CC$ and $p> 0$ we have
\begin{equation}
e^{p,0}(Z^*)_{\alpha}=e^{0,p}(Z^*)_{\overline{\alpha}}= (-1)^{n-1} 
\sum_{\begin{subarray}{c} \Gamma \prec \Delta\\ \d \Gamma =p+1\end{subarray}}l^*(\Gamma)_{\alpha}.
\end{equation}
\end{proposition}

The following result is an analogue of \cite[Corollary 5.10]{D-K}. For $\alpha \in \CC$, denote by $\Pi(\Delta)_{\alpha}$ the number of the lattice points $v=(v_1,\ldots, v_n)$ on the $1$-skeleton of $\Delta^w=\Delta-w$ such that ${\tau}^v=\alpha$, where $w$ is a vertex of $\Delta$.

\begin{proposition}\label{prp:2-19}
{\bf (\cite[Proposition 2.9]{M-T-3})} In the situation as above, for any $\alpha \in \CC^*$ we have
\begin{equation}
e^{0,0}(Z^*)_{\alpha}=
\begin{cases}
(-1)^{n-1} \left(\Pi(\Delta)_{1}-1\right) & (\alpha=1), \\
(-1)^{n-1}  \Pi(\Delta)_{\alpha^{-1}} & (\alpha \neq 1).
\end{cases}
\end{equation}
\end{proposition}

For a vertex $w$ of $\Delta$, we define a closed convex cone $\Con(\Delta, w)$ by $\Con(\Delta,w)=\{ r \cdot (v -w) \ |\ r \in \RR_+, \ v \in \Delta\} \subset \RR^n$.

\begin{definition}\label{dfn:7-6-4}{\bf (\cite{D-K})}
Let $\Delta$ and $\Delta^{\prime}$ be two $n$-dimensional integral polytopes in $(\RR^n, \ZZ^n)$. We denote by ${\rm som}(\Delta )$ (resp. ${\rm som}(\Delta^{\prime})$) the set of vertices of $\Delta$ (resp. $\Delta^{\prime}$). Then we say that $\Delta^{\prime}$ majorizes $\Delta$ if there exists a map $\Psi \colon {\rm som}(\Delta^{\prime}) \longrightarrow {\rm som}(\Delta)$ such that $\Con(\Delta, \Psi(w)) \subset \Con(\Delta^{\prime}, w)$ for any vertex $w$ of $\Delta^{\prime}$.
\end{definition}

For an integral polytope $\Delta$ in $(\RR^n, \ZZ^n)$, we denote by $X_{\Delta}$ the toric variety associated with the dual fan of $\Delta$ (see Fulton \cite{Fulton} and Oda \cite{Oda} etc.). Recall that if $\Delta^{\prime}$ majorizes $\Delta$ there exists a natural morphism $X_{\Delta^{\prime}} \longrightarrow X_{\Delta}$.

\begin{proposition}\label{prp:7-6-5}
{\bf (\cite[Proposition 2.12]{M-T-3})} Let $\Delta$ and $Z^*_{\Delta}=Z^*$ with an action of $l_{\tau}$ be as above. Assume that an $n$-dimensional integral polytope $\Delta^{\prime}$ in $(\RR^n, \ZZ^n)$ majorizes $\Delta$ by the map $\Psi \colon {\rm som}(\Delta^{\prime}) \longrightarrow {\rm som}(\Delta)$. Then for the closure $\overline{Z^*}$ of $Z^*$ in $X_{\Delta^{\prime}}$ we have
\begin{eqnarray}
\sum_q e^{p,q}(\overline{Z^*})_1
&=&\sum_{\Gamma\prec \Delta^{\prime}} (-1)^{\d \Gamma+p+1} \left\{\binom{\d \Gamma}{p+1}-\binom{\dd_{\Gamma}}{p+1}\right\}\nonumber \\
& & +\sum_{\Gamma \prec \Delta^{\prime}}(-1)^{\d \Gamma +1}\sum_{i=0}^{\min\{\dd_{\Gamma},p\}}\binom{\dd_{\Gamma}}{i}(-1)^i \varphi_{1,\d \Psi(\Gamma)-p+i}(\Psi(\Gamma)),\label{eq:7-6-5}
\end{eqnarray}
where for $\Gamma \prec \Delta^{\prime}$ we set $\dd_{\Gamma}=\d \Gamma -\d \Psi(\Gamma)$.
\end{proposition}

\begin{definition}\label{dfn:2-16}
Let $\Delta$ be an $n$-dimensional integral polytope in $(\RR^n, \ZZ^n)$.
\begin{enumerate}
\item (see \cite[Section2.3]{D-K}) We say that $\Delta$ is prime if for any vertex $w$ of $\Delta$ the cone $\Con(\Delta,w)$ is generated by a basis of $\RR^n$.
\item (see \cite[Definition 2.10]{M-T-3}) We say that $\Delta$ is pseudo-prime if for any $1$-dimensional face $\gamma \prec \Delta$ the number of the $2$-dimensional faces $\gamma^{\prime} \prec \Delta$ such that $\gamma \prec \gamma^{\prime}$ is $n-1$.
\end{enumerate}
\end{definition}

By definition, prime polytopes are pseudo-prime. Moreover any face of a pseudo-prime polytope is again pseudo-prime.

For $\alpha \in \CC \setminus \{1\}$ and a face $\Gamma \prec \Delta$, set $\tl{\varphi}_{\alpha}(\Gamma)=\sum_{i=0}^{\d \Gamma} \varphi_{\alpha, i}(\Gamma)$. Then as in \cite[Section 5.5 and Theorem 5.6]{D-K} we obtain the following result. 

\begin{proposition}\label{cor:2-18} 
{\bf (\cite[Corollary 2.15]{M-T-3})} 
Assume that $\Delta=NP(g)$ is pseudo-prime. Then for any $\alpha \in \CC \setminus \{1\}$ and $r \geq 0$, we have
\begin{equation}
\sum_{p+q=r}e^{p,q}(Z^*)_{\alpha}=(-1)^{n+r} \sum_{\begin{subarray}{c} \Gamma \prec \Delta\\ \d \Gamma =r+1\end{subarray}} \left\{ \sum_{\Gamma^{\prime} \prec \Gamma} (-1)^{\d\Gamma^{\prime}}\tl{\varphi}_{\alpha}(\Gamma^{\prime})\right\}.
\end{equation}
\end{proposition}

The following lemma will be used later. 

\begin{lemma}\label{lem:7-6-9}
Let $\gamma$ be a $d$-dimensional prime polytope. Then for any $0 \leq p \leq d$ we have
\begin{equation}\label{eq:7-6-9}
\sum_{\Gamma \prec \gamma} (-1)^{\d \Gamma}\binom{\d \Gamma}{p}=\sum_{\Gamma \prec \gamma}(-1)^{d +\d \Gamma}\binom{\d \Gamma}{d-p}.
\end{equation}
\end{lemma}

\begin{proof}
For a polytope $\Delta$, denote the number of the $j$-dimensional faces of $\Delta$ by $f_{\Delta,j}$ and set $f_{\Delta,-1}=1$. Let $\gamma^{\vee}$ be the dual polytope of $\gamma$. Then $\gamma^{\vee}$ is simplicial and we have $f_{\gamma^{\vee},j}=f_{\gamma, d-1-j}$ for any $0 \leq j \leq d$. Hence \eqref{eq:7-6-9} follows from the Dehn-Sommerville equations (see \cite{Sommerville} etc.) for simplicial polytopes. \qed
\end{proof}

\section{Motivic Milnor fibers}\label{sec:3}

In \cite{D-L-1} and \cite{D-L-2} Denef and Loeser introduced motivic Milnor fibers. In this section, we recall their definition and basic properties. Let $f \in \CC[x_1, x_2, \ldots, x_n]$ be a polynomial such that the hypersurface $f^{-1}(0)= \{x\in \CC^n \ |\ f(x)=0\}$ has an isolated singular point at $0\in \CC^n$. Then by a fundamental theorem of Milnor \cite{Milnor}, for the Milnor fiber $F_0$ of $f$ at $0$ we have $H^j(F_0;\CC) \simeq 0$ ($j\neq 0, \ n-1$). Denote by $\Phi_{n-1,0}\colon H^{n-1}(F_0;\CC) \simto H^{n-1}(F_0;\CC)$ the $(n-1)$-th Milnor monodromy of $f$ at $0 \in \CC^n$. Let $\pi \colon X \longrightarrow \CC^n$ be an embedded resolution of $f^{-1}(0)$ such that $\pi^{-1}(0)$ and $\pi^{-1}(f^{-1}(0))$ are normal crossing divisors in $X$. Let $D_1, D_2, \ldots, D_m$ be the irreducible components of $\pi^{-1}(0)$ and denote by $Z$ the proper transform of $f^{-1}(0)$ in $X$. For $1 \leq i \leq m$ denote by $a_i>0$ the order of the zero of $g:= f \circ \pi$ along $D_i$. For a non-empty subset $I \subset \{ 1,2, \ldots, m\}$ we set $d_I=\gcd (a_i)_{i \in I}>0$, $D_I=\bigcap_{i \in I}D_i$ and
\begin{equation}
D_I^{\circ}=D_I \setminus \left\{ \( \bigcup_{i \notin I}D_i\) \cup Z \right\} \subset X.
\end{equation}
Moreover we set
\begin{equation}
Z_I^{\circ}=\left\{ D_I \setminus \left( \bigcup_{i \notin I}D_i\right) \right\} \cap Z \subset X. 
\end{equation}
Then, as in \cite[Section 3.3]{D-L-2}, we can construct an unramified Galois covering $\tl{D_I^{\circ}} \longrightarrow D_I^{\circ}$ of $D_I^{\circ}$ as follows. First, for a point $p \in D_I^{\circ}$ we take an affine open neighborhood $W \subset X \setminus \{ ( \bigcup_{i \notin I} D_i) \cup Z \}$ of $p$ on which there exist regular functions $\xi_i$ ($i \in I$) such that $D_i \cap W=\{ \xi_i=0 \}$ for any $i \in I$. Then on $W$ we have $g= f \circ \pi =g_{1,W} (g_{2,W})^{d_I}$, where we set $g_{1,W}=g \prod_{i \in I}\xi_i^{-a_i}$ and $g_{2,W}=\prod_{i \in I} \xi_i^{\frac{a_i}{d_I}}$. Note that $g_{1,W}$ is a unit on $W$ and $g_{2,W} \colon W \longrightarrow \CC$ is a regular function. It is easy to see that $D_I^{\circ}$ is covered by such affine open subsets $W$. Then as in \cite[Section 3.3]{D-L-2} by gluing the varieties
\begin{equation}\label{eq:6-26}
\tl{D_{I,W}^{\circ}}=\{(t,z) \in \CC^* \times (D_I^{\circ} \cap W) \ |\ t^{d_I} =(g_{1,W})^{-1}(z)\}
\end{equation}
together in the following way, we obtain the variety $\tl{D_I^{\circ}}$ over $D_I^{\circ}$. If $W^{\prime}$ is another such open subset and $g=g_{1,W^{\prime}} (g_{2,W^{\prime}})^{d_I}$ is the decomposition of $g$ on it, we patch $\tl{D_{I,W}^{\circ}}$ and $\tl{D_{I,W^{\prime}}^{\circ}}$ by the morphism $(t,z) \longmapsto (g_{2,W^{\prime}}(z)( g_{2,W})^{-1}(z) \cdot t, z)$ defined over $W \cap W^{\prime}$. Now for $d \in \ZZ_{>0}$, let $\mu_d \simeq \ZZ/\ZZ d$ be the multiplicative group consisting of the $d$-roots in $\CC$. We denote by $\hat{\mu}$ the projective limit $\underset{d}{\varprojlim} \mu_d$ of the projective system $\{ \mu_i \}_{i \geq 1}$ with morphisms $\mu_{id} \longrightarrow \mu_i$ given by $t \longmapsto t^d$. Then the unramified Galois covering $\tl{D_I^{\circ}}$ of $D_I^{\circ}$ admits a natural $\mu_{d_I}$-action defined by assigning the automorphism $(t,z) \longmapsto (\zeta_{d_I} t, z)$ of $\tl{D_I^{\circ}}$ to the generator $\zeta_{d_I}:=\exp (2\pi\sqrt{-1}/d_I) \in \mu_{d_I}$. Namely the variety $\tl{D_I^{\circ}}$ is equipped with a good $\hat{\mu}$-action in the sense of Denef-Loeser \cite[Section 2.4]{D-L-2}. Note that also the variety $Z_I^{\circ}$ is equipped with the trivial good $\hat{\mu}$-action. Following the notations in \cite{D-L-2}, denote by $\M_{\CC}^{\hat{\mu}}$ the ring obtained from the Grothendieck ring $\KK_0^{\hat{\mu}}(\Var_{\CC})$ of varieties over $\CC$ with good $\hat{\mu}$-actions by inverting the Lefschetz motive $\LL\simeq \CC \in \KK_0^{\hat{\mu}}(\Var_{\CC})$. Recall that $\LL \in \KK_0^{\hat{\mu}}(\Var_{\CC})$ is endowed with the trivial action of $\hat{\mu}$.

\begin{definition}{\bf (Denef and Loeser \cite{D-L-1} 
and \cite{D-L-2})}\label{dfn:3-1} We define the motivic Milnor fiber $\SS_{f,0} \in \M_{\CC}^{\hat{\mu}}$ of $f$ at $0 \in \CC^n$ by
\begin{equation}\label{MMF}
\SS_{f,0} =\sum_{I \neq \emptyset}\left\{ (1-\LL)^{\sharp I -1} [\tl{D_I^{\circ}}] + (1-\LL)^{\sharp I} [Z_I^{\circ}]\right\} \in \M_{\CC}^{\hat{\mu}}.
\end{equation}
\end{definition}

As in \cite[Section 3.1.2 and 3.1.3]{D-L-2}, we denote by $\HSm$ the abelian category of Hodge structures with a quasi-unipotent endomorphism. Let $\KK_0(\HSm)$ be its Grothendieck ring. Then as in \cite{D-L-2}, to the cohomology groups $H^j(F_0;\CC)$ and the semisimple parts of their monodromy automorphisms, we can naturally associate an element
\begin{equation}
[H_f] \in \KK_0(\HSm). 
\end{equation}
To describe the element $[H_f]\in \KK_0(\HSm)$ in terms of $\SS_{f,0} \in \M_{\CC}^{\hat{\mu}}$, let
\begin{equation}
\chi_h \colon \M_{\CC}^{\hat{\mu}} \longrightarrow \KK_0(\HSm)
\end{equation}
be the Hodge characteristic morphism defined in \cite{D-L-2} which associates to a variety $Z$ with a good $\mu_d$-action the Hodge structure
\begin{equation}
\chi_h ([Z])=\sum_{j \in \ZZ} (-1)^j [H_c^j(Z;\QQ)] \in \KK_0(\HSm)
\end{equation}
with the actions induced by the one $z \longmapsto \exp (2\pi\sqrt{-1}/d)z$ ($z\in Z$) on $Z$. Then we have the following fundamental result. 

\begin{theorem}\label{thm:7-6}
{\bf (Denef-Loeser \cite[Theorem 4.2.1]{D-L-1})} In the Grothendieck group $\KK_0(\HSm)$, we have
\begin{equation}
[H_f]=\chi_h(\SS_{f,0}).
\end{equation}
\end{theorem}

For $[H_f] \in \KK_0(\HSm)$ also the following result due to Steenbrink \cite{Steenbrink} and Saito \cite{Saito-1}, \cite{Saito-2} is fundamental.

\begin{theorem}[Steenbrink \cite{Steenbrink} and Saito \cite{Saito-1}, \cite{Saito-2}]\label{S-S}
In the situation as above, we have
\begin{enumerate}
\item Let $\lambda \in \CC^* \setminus \{1\}$. Then we have $e^{p,q}( [H_f])_{\lambda}=0$ for $(p,q) \notin [0,n-1] \times [0,n-1]$. Moreover for $(p,q) \in [0,n-1] \times [0,n-1]$ we have
\begin{equation}
e^{p,q}( [H_f])_{\lambda}=e^{n-1-q,n-1-p}( [H_f])_{\lambda}.
\end{equation}
\item We have $e^{p,q}( [H_f])_{1}=0$ for $(p,q) \notin \{(0, 0)\} \sqcup ([1,n-1] \times [1,n-1])$ and $e^{0,0}( [H_f])_{1}=1$. Moreover for $(p,q) \in [1,n-1] \times [1,n-1]$ we have
\begin{equation}
e^{p,q}( [H_f])_{1}=e^{n-q,n-p}( [H_f])_{1}.
\end{equation}
\end{enumerate}
\end{theorem}

We can check these symmetries of $e^{p,q}([H_f])_{\lambda}$ by calculating $\chi_h( \SS_{f,0}) \in \KK_0(\HSm)$ explicitly by our methods (see Section \ref{sec:4}) in many cases. Since the weights of $[H_f] \in \KK_0(\HSm)$ are defined by the monodromy filtration, we have the following result.

\begin{theorem}\label{MF}
In the situation as above, we have
\begin{enumerate}
\item Let $\lambda \in \CC^* \setminus \{1\}$ and $k \geq 1$. Then the number of the Jordan blocks for the eigenvalue $\lambda$ with sizes $\geq k$ in $\Phi_{n-1,0}\colon H^{n-1}(F_0;\CC) \simto H^{n-1}
(F_0;\CC)$ is equal to
\begin{equation}
(-1)^{n-1} \sum_{p+q=n-2+k, n-1+k} e^{p,q}( \chi_h(\SS_{f,0} ))_{\lambda}.
\end{equation}
\item For $k \geq 1$, the number of the Jordan blocks for the eigenvalue $1$ with sizes $\geq k$ in $\Phi_{n-1, 0}$ is equal to
\begin{equation}
(-1)^{n-1} \sum_{p+q=n-1+k, n+k} e^{p,q}( \chi_h(\SS_{f,0} ))_{1}.
\end{equation}
\end{enumerate}
\end{theorem}

\section{Jordan normal forms of Milnor monodromies}\label{sec:4}

Our methods in \cite{M-T-3} can be applied also to the Jordan normal forms of local Milnor monodromies. Let $f\in \CC[x_1,\ldots,x_n]$ be a polynomial such that the hypersurface $\{x\in \CC^n \ |\ f(x)=0\}$ has an isolated singular point at $0\in \CC^n$. 

\begin{definition}\label{CVN} 
Let $f(x)= \sum_{v \in \ZZ_+^n} a_v x^v \in \CC[x_1,\ldots,x_n]$ be a polynomial on $\CC^n$.
\begin{enumerate}
\item We call the convex hull of $\bigcup_{v \in \supp (f)} \{v + \RR_+^n\}$ in $\RR_+^n$ the Newton polyhedron of $f$ and denote it by $\Gamma_+(f)$.
\item The union of the compact faces of $\Gamma_+(f)$ is called the Newton boundary of $f$ and denoted by $\Gamma_f$.
\item
We say that $f$ is convenient if $\Gamma_+(f)$ intersects the positive part of any coordinate axis in $\RR^n$. 
\end{enumerate}
\end{definition}

\begin{definition}[\cite{Kushnirenko}]\label{NDC} 
We say that a polynomial $f(x)=\sum_{v \in \ZZ_+^n}a_vx^v$ ($a_v\in \CC$) is non-degenerate at $0\in \CC^n$ if for any face $\gamma \prec \Gamma_+(f)$ such that $\gamma \subset \Gamma_f$ the complex hypersurface $\{x \in (\CC^*)^n\ |\ f_{\gamma}(x)=0\}$ in $(\CC^*)^n$ is smooth and reduced, where we set $f_{\gamma}(x)=\sum_{v \in \gamma \cap \ZZ_+^n} a_vx^v$.
\end{definition}

Recall that generic polynomials having a fixed Newton polyhedron are non-degenerate at $0\in \CC^n$. From now on, we always assume also that $f=\sum_{v \in \ZZ^n_+} a_v x^v\in \CC[x_1,\ldots,x_n]$ is convenient and non-degenerate at $0\in \CC^n$. For each face $\gamma \prec \Gamma_+(f)$ such that $\gamma \subset \Gamma_f$, let $d_{\gamma}>0$ be the lattice distance of $\gamma$ from the origin $0 \in \RR^n$ and $\Delta_{\gamma}$ the convex hull of $\{0\} \sqcup \gamma$ in $\RR^n$. Let $\LL(\Delta_{\gamma})$ be the $(\dim \gamma +1)$-dimensional linear subspace of $\RR^n$ spanned by $\Delta_{\gamma}$ and consider the lattice $M_{\gamma}=\ZZ^n \cap \LL(\Delta_{\gamma}) \simeq \ZZ^{\dim \gamma+1}$ in it. Then we set $T_{\Delta_{\gamma}}:=\Spec (\CC[M_{\gamma}]) \simeq (\CC^*)^{\dim \gamma +1}$. Moreover let $\LL(\gamma)$ be the smallest affine linear subspace of $\RR^n$ containing $\gamma$ and for $v \in M_{\gamma}$ define their lattice heights $\height (v, \gamma) \in \ZZ$ from $\LL(\gamma)$ in $\LL(\Delta_{\gamma})$ so that we have $\height (0, \gamma)=d_{\gamma}>0$. Then to the group homomorphism $M_{\gamma} \longrightarrow \CC^*$ defined by $v \longmapsto \zeta_{d_{\gamma}}^{-\height (v, \gamma)}$ we can naturally associate an element $\tau_{\gamma} \in T_{\Delta_{\gamma}}$. We define a Laurent polynomial $g_{\gamma}=\sum_{v \in M_{\gamma}}b_v x^v$ on $T_{\Delta_{\gamma}}$ by
\begin{equation}
b_v=\begin{cases}
a_v & (v \in \gamma),\\
-1 & (v=0),\\
\ 0 & (\text{otherwise}).
\end{cases}
\end{equation}
Then we have $NP(g_{\gamma}) =\Delta_{\gamma}$, $\supp (g_{\gamma}) \subset \{ 0\} \sqcup \gamma$ and the hypersurface $Z_{\Delta_{\gamma}}^*=\{ x \in T_{\Delta_{\gamma}}\ |\ g_{\gamma}(x)=0\}$ is non-degenerate by \cite[Proposition 5.3]{M-T-3}. Moreover $Z_{\Delta_{\gamma}}^* \subset T_{\Delta_{\gamma}}$ is invariant by the multiplication $l_{\tau_{\gamma}} \colon T_{\Delta_{\gamma}} \simto T_{\Delta_{\gamma}}$ by $\tau_{\gamma}$, and hence we obtain an element $[Z_{\Delta_{\gamma}}^*]$ of $\M_{\CC}^{\hat{\mu}}$. Let $\LL(\gamma)^{\prime} \simeq \RR^{\d \gamma}$ be a linear subspace of $\RR^n$ such that $\LL(\gamma)=\LL(\gamma)^{\prime}+w $ for some $w \in \ZZ^n$ and set $\gamma^{\prime}=\gamma -w \subset \LL(\gamma)^{\prime}$. We define a Laurent polynomial $g_{\gamma}^{\prime} =\sum_{v \in \LL(\gamma)^{\prime} \cap \ZZ^n}b_v^{\prime} x^v$ on $T(\gamma ):= \Spec (\CC[\LL(\gamma)^{\prime} \cap \ZZ^n])\simeq (\CC^*)^{\d \gamma}$ by
\begin{equation}
b_v^{\prime}=\begin{cases}
a_{v+w} & (v \in \gamma^{\prime} ),\\
\ 0 & (\text{otherwise}).
\end{cases}
\end{equation}
Then we have $NP(g_{\gamma}^{\prime}) =\gamma^{\prime}$ and the hypersurface $Z_{\gamma}^*=\{ x \in T(\gamma ) \ |\ g_{\gamma}^{\prime}(x)=0\}$ is non-degenerate. We define $[Z_{\gamma}^*] \in \M_{\CC}^{\hat{\mu}}$ to be the class of the variety $Z_{\gamma}^*$ with the trivial action of $\hat{\mu}$. Finally let $S_{\gamma} \subset \{1,2,\ldots, n\}$ be the minimal subset $S$ of $\{1,2,\ldots,n\}$ such that $\gamma \subset \{ (y_1, y_2, \ldots, y_n) \in \RR^n \ | \ y_i=0 \quad \text{for any} \ i \notin S \} \simeq \RR^{\sharp S}$ and set $m_{\gamma}:=\sharp S_{\gamma}-\dim \gamma -1\geq 0$. Then as in the same way as \cite[Theorem 5.7]{M-T-3} we obtain the following theorem. 

\begin{theorem}\label{thm:8-3}
In the situation as above, we have
\begin{enumerate}
\item In the Grothendieck group $\KK_0(\HSm)$, we have
\begin{equation}\label{twot}
\chi_h(\SS_{f,0})= \sum_{\gamma \subset \Gamma_f} \chi_h\big((1-\LL)^{m_{\gamma}}\cdot[Z_{\Delta_{\gamma}}^*]\big)+\sum_{\begin{subarray}{c}\gamma \subset \Gamma_f\\ \d \gamma \geq 1\end{subarray}} \chi_h\big((1-\LL)^{m_{\gamma}+1}\cdot[Z_{\gamma}^*]\big).
\end{equation}
\item Let $\lambda \in \CC^*\setminus\{1\}$ and $k\geq 1$. Then the number of the Jordan blocks for the eigenvalue $\lambda$ with sizes $\geq k$ in $\Phi_{n-1,0}\colon H^{n-1}(F_0;\CC) \simto H^{n-1}(F_0;\CC)$ is equal to
\begin{equation}
(-1)^{n-1}\sum_{p+q=n-2+k, n-1+k}\left\{ \sum_{\gamma \subset \Gamma_f} e^{p,q}\( \chi_h\((1-\LL)^{m_{\gamma}} \cdot [Z_{\Delta_{\gamma}}^*]\)\)_{\lambda} \right\}.
\end{equation}
\item For $k\geq 1$, the number of the Jordan blocks for the eigenvalue $1$ with sizes $\geq k$ in $\Phi_{n-1,0}$ is equal to
\begin{eqnarray}
(-1)^{n-1}\sum_{p+q=n-1+k, n+k}\lefteqn{\Bigg\{ \sum_{\gamma \subset \Gamma_f} e^{p,q}\big( \chi_h\big((1-\LL)^{m_{\gamma}} \cdot [Z_{\Delta_{\gamma}}^*]\big)\big)_{1}} \nonumber \\
& & +\sum_{\begin{subarray}{c}\gamma \subset \Gamma_f\\ \d \gamma \geq 1\end{subarray}} e^{p,q}\big( \chi_h\big((1-\LL)^{m_{\gamma}+1} \cdot [Z_{\gamma}^*]\big)\big)_{1} \Bigg\}.
\end{eqnarray}
\end{enumerate}
\end{theorem}

\begin{proof}
Since (ii) and (iii) follow from (i) and Theorem \ref{MF}, it suffices to prove (i). The proof is very similar to the one in Varchenko \cite{Varchenko}. Let $\Sigma_1$ be the dual fan of $\Gamma_+(f)$ in $\RR_+^n$ and $\Sigma$ its smooth subdivision. Denote by $X_{\Sigma}$ the smooth toric variety associated to $\Sigma$ (see Fulton \cite{Fulton} and Oda \cite{Oda} etc.). Since the union of the cones in $\Sigma$ is $\RR_+^n$, there exists a proper morphism $\pi \colon X_{\Sigma} \longrightarrow \CC^n$. By the convenience of $f$, we can construct the smooth fan $\Sigma$ without subdividing the cones contained in $\partial \RR_+^n$ (see \cite[Lemma (2.6), Chapter II]{Oka}). Then $\pi$ induces an isomorphism $X_{\Sigma} \setminus \pi^{-1}(0) \simeq \CC^n \setminus \{ 0\}$. Moreover by the non-degeneracy at $0 \in \CC^n$ of $f$, the proper transform $Z$ of the hypersurface $\{x\in \CC^n \ |\ f(x)=0\}$ in $X_{\Sigma}$ is smooth and intersects $T$-orbits in $\pi^{-1}(0)$ transversally. Let $D_1, \ldots, D_m$ be the toric divisors in $\pi^{-1}(0) \subset X_{\Sigma}$. For a non-empty subset $I \subset \{ 1,2, \ldots, m\}$ we set $D_I=\bigcap_{i \in I}D_i$ and
\begin{equation}
D_I^{\circ}=D_I \setminus \left\{ \( \bigcup_{i \notin I}D_i\) \cup Z \right\} \subset X_{\Sigma}
\end{equation}
and define its unramified Galois covering $\tl{D_I^{\circ}}$ as in Section \ref{sec:3}. Moreover we set
\begin{equation}
Z_I^{\circ}=\left\{ D_I \setminus \left( \bigcup_{i \notin I}D_i\right) \right\} \cap Z \subset X_{\Sigma}
\end{equation}
and denote by $[Z_I^{\circ}] \in \M_{\CC}^{\hat{\mu}}$ the class of the variety $Z_I^{\circ}$ with the trivial action. Then, unlike the global object $\SS_f^{\infty}$ in \cite{M-T-3}, Denef-Loeser's ``local" motivic Milnor fiber $\SS_{f,0}$ contains not only $(1-\LL)^{\sharp I -1} [\tl{D_I^{\circ}}]$ but also $(1-\LL)^{\sharp I} [Z_I^{\circ}]$ (see Definition \ref{dfn:3-1}). These new elements yield the second term in the right hand side of \eqref{twot}. Finally, in the Grothendieck group $\KK_0(\HSm)$ we can rewrite $\chi_h(\SS_{f,0})$ in terms of the dual fan $\Sigma_1$ (i.e. in terms of $\Gamma_+(f)$) as in the same way as the proof of \cite[Theorem 5.7 (i)]{M-T-3}. This completes the proof. \qed
\end{proof}

Let $q_1,\ldots, q_l$ (resp. $\gamma_1, \ldots, \gamma_{l^{\prime}}$) be the $0$-dimensional (resp. $1$-dimensional) faces of $\Gamma_+(f)$ such that $q_i \in \Int (\RR_+^n)$ (resp. $\relint(\gamma_i) \subset \Int(\RR_+^n)$). Here $\relint (\cdot )$ stands for the relative interior. For each $q_i$ (resp. $\gamma_i$), denote by $d_i >0$ (resp. $e_i>0$) the lattice distance $\dist(q_i, 0)$ (resp. $\dist(\gamma_i,0)$) of it from the origin $0\in \RR^n$. For $1\leq i \leq l^{\prime}$, let $\Delta_i$ be the convex hull of $\{0\}\sqcup \gamma_i$ in $\RR^n$. Then for $\lambda \in \CC \setminus \{1\}$ and $1 \leq i \leq l^{\prime}$ such that $\lambda^{e_i}=1$ we set
\begin{equation}
n(\lambda)_i
= \sharp\{ v\in \ZZ^n \cap \relint(\Delta_i) \ |\ \height (v, \gamma_i)=k\} +\sharp \{ v\in \ZZ^n \cap \relint(\Delta_i) \ |\ \height (v, \gamma_i)=e_i-k\},
\end{equation}
where $k$ is the minimal positive integer satisfying $\lambda=\zeta_{e_i}^{k}$ and for $v\in \ZZ^n \cap \relint(\Delta_i)$ we denote by $\height (v, \gamma_i)$ the lattice height of $v$ from the base $\gamma_i$ of $\Delta_i$. As in the same way as \cite[Theorem 5.9]{M-T-3}, by using Propositions \ref{prp:new} and \ref{prp:2-19} and Theorem \ref{thm:8-3} (ii), we obtain the following theorem. 

\begin{theorem}\label{thm:8-4}
In the situation as above, for $\lambda \in \CC^* \setminus\{1\}$, we have
\begin{enumerate}
\item The number of the Jordan blocks for the eigenvalue $\lambda$ with the maximal possible size $n$ in $\Phi_{n-1,0}$ is equal to $\sharp \{q_i \ |\ \lambda^{d_i}=1\}$.
\item The number of the Jordan blocks for the eigenvalue $\lambda$ with size $n-1$ in $\Phi_{n-1,0}$ is equal to $\sum_{i \colon \lambda^{e_i}=1} n(\lambda)_i$.
\end{enumerate}
\end{theorem}

Note that by Theorem \ref{thm:8-3} and our results in Section \ref{sec:2} we can always calculate the whole Jordan normal form of $\Phi_{n-1,0}$. From now on, we shall rewrite Theorem \ref{thm:8-3} (ii) more explicitly in the case where any face $\gamma \prec \Gamma_{+}(f)$ such that $\gamma \subset \Gamma_f$ is prime (see Definition \ref{dfn:2-16} (i)). Recall that by Proposition \ref{prp:2-15} for $\lambda \in \CC^* \setminus \{1\}$ and a face $\gamma \prec \Gamma_{+}(f)$ such that $\gamma \subset \Gamma_f$ we have $e^{p,q}(Z_{\Delta_{\gamma}}^*)_{\lambda}=0$ for any $p,q \geq 0$ such that $p+q > \d \Delta_{\gamma}-1=\dim \gamma$. So the non-negative integers $r \geq 0$ such that $\sum_{p+q=r}e^{p,q}(Z_{\Delta_{\gamma}}^*)_{\lambda}\neq 0$ are contained in the closed interval $[0,\d \gamma]\subset \RR$.

\begin{definition}
For a face $\gamma \prec \Gamma_{+}(f)$ such that $\gamma \subset \Gamma_f$ and $k \geq 1$, we define a finite subset $J_{\gamma,k} \subset [0,\d \gamma] \cap \ZZ$ by
\begin{equation}
J_{\gamma,k}=\{0 \leq r\leq \d \gamma \ |\ n-2+k \equiv r \mod 2\}.
\end{equation}
For each $r\in J_{\gamma,k}$, set
\begin{equation}
d_{k,r}=\dfrac{n-2+k-r}{2}\in \ZZ_+.
\end{equation}
\end{definition}

If a face $\gamma \prec \Gamma_{+}(f)$ such that $\gamma \subset \Gamma_f$ is prime, then the polytope $\Delta_{\gamma}$ is pseudo-prime (see Definition \ref{dfn:2-16} (ii)). Then by Proposition \ref{cor:2-18} for $\lambda \in \CC^* \setminus \{1\}$ and an integer $r \geq 0$ such that $r\in [0,\d \gamma] $ we have
\begin{equation}
\sum_{p+q=r}e^{p,q}(\chi_h([Z_{\Delta_{\gamma}}^*]))_{\lambda}=(-1)^{\d \gamma +r+1} \sum_{\begin{subarray}{c} \Gamma\prec \Delta_{\gamma} \\ \d \Gamma=r+1\end{subarray}} \left\{ \sum_{\Gamma^{\prime} \prec \Gamma} (-1)^{\d \Gamma^{\prime}} \tl{\varphi}_{\lambda}(\Gamma^{\prime})\right\}.
\end{equation}
For simplicity, we denote this last integer by $e(\gamma,\lambda)_r$. Then by Theorem \ref{thm:8-3} (ii) we obtain the following result.

\begin{theorem}\label{thm:7-15} 
Assume that any face $\gamma \prec \Gamma_{+}(f)$ such that $\gamma \subset \Gamma_f$ is prime. Let $\lambda \in \CC^* \setminus\{1\}$ and $k\geq 1$. Then the number of the Jordan blocks for the eigenvalue $\lambda$ with sizes $\geq k$ in $\Phi_{n-1,0} \colon H^{n-1}(F_0;\CC) \simto H^{n-1}(F_0;\CC)$ is equal to 
\begin{equation}
(-1)^{n-1}\sum_{\gamma \subset \Gamma_f} \left\{ \sum_{r \in J_{\gamma, k}} (-1)^{d_{k,r}} \binom{m_{\gamma}}{d_{k,r}} \cdot e(\gamma,\lambda)_r + \sum_{r \in J_{\gamma, k+1}} (-1)^{d_{k+1,r}} \binom{m_{\gamma}}{d_{k+1,r}} \cdot e(\gamma,\lambda)_r\right\},
\end{equation}
where we used the convention $\binom{a}{b}=0$ ($0 \leq a <b$) for binomial coefficients.
\end{theorem}

By combining the proof of \cite[Theorem 5.6]{D-K} and \cite[Proposition 2.14]{M-T-3} with Theorem \ref{thm:8-3} (iii), if any face $\gamma \prec \Gamma_{+}(f)$ such that $\gamma \subset \Gamma_f$ is prime we can also describe the Jordan blocks for the eigenvalue $1$ in $\Phi_{n-1,0}$ by a closed formula. Since this result is rather involved, we omit it here. 

\begin{remark}
Our results above are different from the previous ones due to Danilov \cite{Danilov} and Tanab{\'e} \cite{Tanabe}. For example, in \cite{Danilov} and \cite{Tanabe} they assume a stronger condition that the Newton polyhedron $\Gamma_+(f)$ itself is prime. We could weaken their condition, because our \cite[Propositions 2.13 and 2.14]{M-T-3} and Proposition \ref{cor:2-18} are generalizations of the corresponding results in \cite{D-K} to pseudo-prime polytopes.
\end{remark}

We can also obtain the corresponding results for the eigenvalue $1$ by rewriting Theorem \ref{thm:8-3} (iii) more simply as follows.

\begin{theorem}\label{thm:8}
In the situation of Theorem \ref{thm:8-3}, for $k\geq 1$ the number of the Jordan blocks for the eigenvalue $1$ with sizes $\geq k$ in $\Phi_{n-1,0}$ is equal to
\begin{equation}
(-1)^{n-1}\sum_{p+q=n-2-k, n-1-k}\left\{ \sum_{\gamma \subset \Gamma_f} e^{p,q}\( \chi_h\((1-\LL)^{m_{\gamma}} \cdot [Z_{\Delta_{\gamma}}^*]\)\)_{1} \right\}.
\end{equation}
\end{theorem}

As in the same way as \cite[Theorems 5.11 and 5.12]{M-T-3}, by using Propositions \ref{prp:new} and \ref{prp:2-19} and Theorem \ref{thm:8}, we obtain the following corollary. Denote by $\Pi_f$ the number of the lattice points on the $1$-skeleton of $\Gamma_f \cap \Int(\RR_+^n)$. Also, for a compact face $\gamma \prec \Gamma_+(f)$ we denote by $l^*(\gamma)$ the number of the lattice points on $\relint(\gamma)$.

\begin{corollary}\label{thm:8-5}
In the situation as above, we have
\begin{enumerate}
\item {\rm (van Doorn-Steenbrink \cite{D-St})} The number of the Jordan blocks for the eigenvalue $1$ with the maximal possible size $n-1$ in $\Phi_{n-1,0}$ is $\Pi_f$.
\item The number of the Jordan blocks for the eigenvalue $1$ with size $n-2$ in $\Phi_{n-1,0}$ is equal to $2\sum_{\gamma} l^*(\gamma)$, where $\gamma$ ranges through the compact faces of $\Gamma_+(f)$ such that $\d \gamma=2$ and $\relint(\gamma) \subset \Int(\RR_+^n)$.
\end{enumerate}
\end{corollary}

Note that Corollary \ref{thm:8-5} (i) was previously obtained in van Doorn-Steenbrink \cite{D-St} by different methods. Theorem \ref{thm:8} asserts that by replacing $\Gamma_+(f)$ with the Newton polyhedron at infinity $\Gamma_{\infty}(f)$ in \cite{L-S}, \cite{M-T-2} and \cite{M-T-3} etc. the combinatorial description of the local monodromy $\Phi_{n-1,0}$ is the same as that of the global one $\Phi_{n-1}^{\infty}$ obtained in \cite[Theorem 5.7 (iii)]{M-T-3}. Namely we find a beautiful symmetry between local and global. Theorem \ref{thm:8} can be deduced from the following more precise result.

\begin{theorem}\label{thm:7-6-1}
In the situation as above, for any $0 \leq p,q \leq n-2$ we have
\begin{eqnarray}
\lefteqn{\sum_{\gamma \subset \Gamma_f}e^{p,q}\( \chi_h\( (1-\LL)^{m_{\gamma}}[Z_{\Delta_{\gamma}}^*]\)\)_1}\nonumber \\
&=&\sum_{\gamma \subset \Gamma_f}e^{p+1,q+1}
\( \chi_h\( (1-\LL)^{m_{\gamma}}[Z_{\Delta_{\gamma}}^*]+ (1-\LL)^{m_{\gamma}+1}[Z_{\gamma}^*]\)\)_1.
\label{eq:7-6-1}
\end{eqnarray}
\end{theorem}

We can easily see that Theorem \ref{thm:7-6-1} follows from Proposition \ref{prp:7-6-2} below. For $[V] \in \KK_0(\HSm)$, let $e([V])_1=\sum_{p,q=0}^{\infty}e^{p,q}([V])_1t_1^pt_2^q$ be the generating function of $e^{p,q}([V])_1$ as in \cite{D-K}.

\begin{proposition}\label{prp:7-6-2}
We have
\begin{equation}
\sum_{\gamma \subset \Gamma_f}e\( \chi_h\( (1-\LL)^{m_{\gamma}+1}([Z_{\Delta_{\gamma}}^*]+[Z_{\gamma}^*])\)\)_1=1-(t_1t_2)^n.
\end{equation}
\end{proposition}

From now on, we shall prove Proposition \ref{prp:7-6-2}. First, we apply Proposition \ref{prp:7-6-5} to the case where $\Delta =\Delta_{\gamma}$ for a face $\gamma$ of $\Gamma_+(f)$ such that $\gamma \subset \Gamma_f$. Let $\gamma^{\prime}$ be a prime polytope in $\RR^{\d \gamma}$ which majorizes $\gamma$ and consider the Minkowski sum $\gamma^{\prime \prime}:=\gamma+\gamma^{\prime}$ (resp. $\Box_{\gamma^{\prime\prime}}:= \Delta_{\gamma}+\gamma^{\prime}$) in $\RR^{\d \gamma}$ (resp. $\RR^{\d \gamma +1}$). Then $\Box_{\gamma^{\prime\prime}}$ is a $(\d \gamma +1)$-dimensional truncated pyramid whose top (resp. bottom) is $\gamma^{\prime}$ (resp. $\gamma^{\prime \prime}$) (see Figure 1 below). In particular, $\Box_{\gamma^{\prime\prime}}$ is prime. Since the dual fan of $\gamma^{\prime \prime}$ coincides with that of $\gamma^{\prime}$, the prime polytope $\gamma^{\prime \prime}$ majorizes $\gamma$. Let $\Psi \colon {\rm som}(\gamma^{\prime \prime}) \longrightarrow {\rm som}(\gamma)$ be the morphism between the sets of the vertices of $\gamma^{\prime \prime}$ and $\gamma$. By extending $\Psi$ to a morphism $\tl{\Psi} \colon {\rm som}(\Box_{\gamma^{\prime \prime}}) \longrightarrow {\rm som}(\Delta_{\gamma})$ as
\begin{equation}
\tl{\Psi}(w)=
\begin{cases}
\Psi(w) & (w\in {\rm som}(\gamma^{\prime \prime})),\\
\{0\} & (w\in {\rm som}(\gamma^{\prime})),
\end{cases}
\end{equation}
we see that the prime polytope $\Box_{\gamma^{\prime \prime}}$ majorizes $\Delta_{\gamma}$.

\vspace{3mm}
\begin{center}
\includegraphics[scale=0.75]{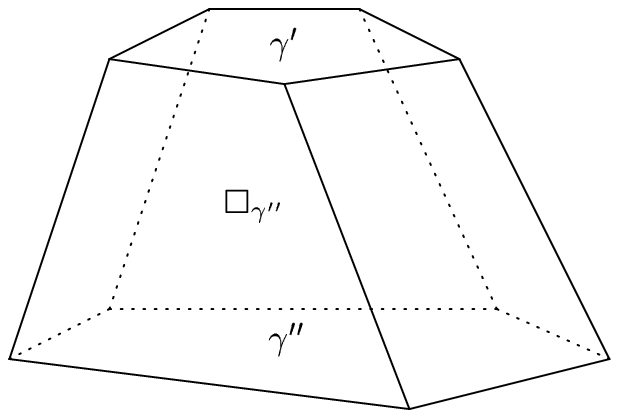}

Figure 1
\end{center}

\begin{proposition}\label{cor:7-6-6}
For the closure $\overline{Z_{\Delta_{\gamma}}^*}$ of $Z_{\Delta_{\gamma}}^*$ in $X_{\Box_{\gamma^{\prime \prime}}}$, we have
\begin{equation}
\sum_q e^{p,q}(\overline{Z_{\Delta_{\gamma}}^*})_1=\sum_{\tau\prec \gamma^{\prime \prime}} (-1)^{\d \tau+p} \binom{\d \tau}{p}.
\end{equation}
\end{proposition}

\begin{proof}
It suffices to rewrite Proposition \ref{prp:7-6-5} in this case. For a face $\Gamma$ of $\Box_{\gamma^{\prime\prime}}$, we set $\dd_{\Gamma}=\d \Gamma - \d \tl{\Psi} (\Gamma)$. Note that the set of faces of $\Box_{\gamma^{\prime \prime}}$ consists of those of $\gamma^{\prime}$ and $\gamma^{\prime \prime}$ and side faces. Each side face of $\Box_{\gamma^{\prime\prime}}$ is a truncated pyramid $\Box_{\tau}$ whose bottom is $\tau\prec \gamma^{\prime \prime}$. Since $\d \Box_{\tau}=\d \tau +1$ and $\dd_{\Box_{\tau}}=\dd_{\tau}$ for $\tau \prec \gamma^{\prime \prime}$, we have
\begin{equation}
\sum_{\Gamma\prec \Box_{\gamma^{\prime \prime}}} (-1)^{\d\Gamma+p+1} \left\{\binom{\d \Gamma}{p+1}-\binom{\dd_{\Gamma}}{p+1}\right\}= \sum_{\tau \prec \gamma^{\prime \prime}}(-1)^{\d \tau +p} \binom{\d \tau}{p}
\end{equation}
and
\begin{eqnarray}
\lefteqn{\sum_{\Gamma \prec \Box_{\gamma^{\prime \prime}}}(-1)^{\d \Gamma +1}\sum_{i=0}^{\min\{\dd_{\Gamma},p\}}\binom{\dd_{\Gamma}}{i}(-1)^i \varphi_{1,\d \tl{\Psi}(\Gamma)-p+i}(\tl{\Psi}(\Gamma))}\nonumber \\
&=& \sum_{\tau \prec \gamma^{\prime \prime}}(-1)^{\d \tau +1}\sum_{i=0}^{\min\{\dd_{\tau},p\}}\binom{\dd_{\tau}}{i}(-1)^i \nonumber \\
& & \hspace{10mm}\times \left\{\varphi_{1,\d \Psi(\tau)-p+i}(\Psi(\tau))-\varphi_{1, \d \tl{\Psi}(\Box_{\tau})-p+i}(\tl{\Psi}(\Box_{\tau}))\right\},
\end{eqnarray}
where the faces $\tau$ of the top $\gamma^{\prime}$ of $\Box_{\gamma^{\prime \prime}}$ are neglected by the condition $\d \tl{\Psi}(\tau)=0$. By $\tl{\Psi}(\Box_{\tau})=\Delta_{\Psi(\tau)}$ and Lemma \ref{lem:7-6-7} below, the last term is equal to $0$.\qed
\end{proof}

\begin{lemma}\label{lem:7-6-7}
For any face $\gamma$ of $\Gamma_+(f)$ such that $\gamma \subset \Gamma_f$, we have
\begin{equation}\label{eq:7-6-7}
\varphi_{1,j+1}(\Delta_{\gamma})=\varphi_{1,j}(\gamma).
\end{equation}
\end{lemma}

\begin{proof}
By the relation $l^*((k+1)\Delta_{\gamma})_1-l^*(k\Delta_{\gamma})_1=l^*(k\gamma)_1$ ($k \geq 0$) we have
\begin{equation}
P_1(\Delta_{\gamma};t)=tP_1(\gamma;t).
\end{equation}
By comparing the coefficients of $t^{j+1}$ in both sides, we obtain \eqref{eq:7-6-7}.\qed
\end{proof}

The following proposition is a key in the proof of Proposition \ref{prp:7-6-2}.

\begin{proposition}\label{thm:7-6-10}
For any face $\gamma$ of $\Gamma_+(f)$ such that $\gamma \subset \Gamma_f$, we have
\begin{equation}\label{eq:7-6-10}
e(\chi_h([Z_{\Delta_{\gamma}}^*]+[Z_{\gamma}^*]))_1=(t_1t_2-1)^{\d \gamma}.
\end{equation}
\end{proposition}

\begin{proof}
It is enough to prove
\begin{equation}\label{eq:7-6-10-1}
e^{p,q}(Z_{\gamma}^*)_1 +e^{p,q}(Z_{\Delta_{\gamma}}^*)_1=(-1)^{\d \gamma +p}\binom{\d \gamma}{p} \cdot \delta_{p,q},
\end{equation}
where $\delta_{p,q}$ is Kronecker's delta. We consider the closure $\overline{Z_{\Delta_{\gamma}}^*}$ of $Z_{\Delta_{\gamma}}^*$ in $X_{\Box_{\gamma^{\prime \prime}}}$. Then by the proofs of Propositions \ref{prp:7-6-5} and \ref{cor:7-6-6}, we have
\begin{align}
e^{p,q}(\overline{Z_{\Delta_{\gamma}}^*})_1
&= \sum_{\tau \prec \gamma^{\prime \prime}}\left\{e^{p,q}((\CC^*)^{
\dd_{\tau}}\times Z_{\Psi(\tau)}^*)_1 +e^{p,q}((\CC^*)^{\dd_{\Box_{\tau}}}\times Z_{\tl{\Psi}(\Box_{\tau})}^*)_1\right\}\\
&= \sum_{\tau \prec \gamma^{\prime \prime}} \sum_{i=0}^{\min\{\dd_{\tau},p\}}\binom{\dd_{\tau}}{i}(-1)^{i+\dd_{\tau}} \left\{e^{p-i,q-i}(Z_{\Psi(\tau)}^*)_1 +e^{p-i,q-i}(Z_{\Delta_{\Psi(\tau)}}^*)_1\right\}.
\label{eq:7-6-10-2}
\end{align}

Let us prove \eqref{eq:7-6-10-1} by induction on $\d \gamma$. In the case $\d \gamma=0$, we can prove \eqref{eq:7-6-10-1} easily by Propositions \ref{prp:2-15} and \ref{prp:2-19}. Assume that for any $\sigma\subset \Gamma_f$ such that $\d \sigma <\d \gamma$ \eqref{eq:7-6-10-1} holds. Then by $\dd_{\gamma^{\prime \prime}}=0$ and \eqref{eq:7-6-10-2} we have
\begin{equation}\label{eq:7-6-10-4}
e^{p,q}(\overline{Z_{\Delta_{\gamma}}^*})_1=e^{p,q}(Z_{\gamma}^*)_1 +e^{p,q}(Z_{\Delta_{\gamma}}^*)_1+ \delta_{p,q}\sum_{\tau \precneqq \gamma^{\prime \prime}} (-1)^{\d \tau +p}\binom{\d \tau}{p}.
\end{equation}
In the case $p+q> \d \gamma$, by Proposition \ref{prp:2-15} we have
\begin{equation}
e^{p,q}(\overline{Z_{\Delta_{\gamma}}^*})_1= \delta_{p,q}\sum_{\tau \prec \gamma^{\prime \prime}} (-1)^{\d \tau +p}\binom{\d \tau}{p}.
\end{equation}
Therefore, also in the case $p+q< \d \gamma$, by the Poincar{\'e} duality for $\overline{Z_{\Delta_{\gamma}}^*}$ ($\Box_{\gamma^{\prime \prime}}$ is prime) and Lemma \ref{lem:7-6-9} we have
\begin{eqnarray}
e^{p,q}(\overline{Z_{\Delta_{\gamma}}^*})_1
&=& e^{\d \gamma -p, \d \gamma-q}(\overline{Z_{\Delta_{\gamma}}^*})_1\\
&=& \delta_{p,q}\sum_{\tau \prec \gamma^{\prime \prime}} (-1)^{\d \tau +\d \gamma -p}\binom{\d \tau}{\d \gamma -p}\\
&=& \delta_{p,q}\sum_{\tau \prec \gamma^{\prime \prime}} (-1)^{\d \tau +p}\binom{\d \tau}{p}.
\end{eqnarray}
In the case $p+q=\d \gamma$, by Proposition \ref{cor:7-6-6} and the previous results we have
\begin{eqnarray}
e^{p,q}(\overline{Z_{\Delta_{\gamma}}^*})_1
&=&\sum_{q^{\prime}}e^{p,q^{\prime}}(\overline{Z_{\Delta_{\gamma}}^*})_1-(1-\delta_{p,q})e^{p,p}(\overline{Z_{\Delta_{\gamma}}^*})_1\\
&=&\delta_{p,q} \sum_{\tau \prec \gamma^{\prime \prime}} (-1)^{\d \tau +p}\binom{\d \tau}{p}.
\end{eqnarray}
By \eqref{eq:7-6-10-4}, we obtain \eqref{eq:7-6-10-1} for any $p,q$. \qed
\end{proof}

Now we can finish the proof of Proposition \ref{prp:7-6-2} as follows. By Proposition \ref{thm:7-6-10}, we have
\begin{align}
\sum_{\gamma \subset \Gamma_f}e\( \chi_h\( (1-\LL)^{m_{\gamma}+1}([
Z_{\Delta_{\gamma}}^*]+[Z_{\gamma}^*])\)\)_1
&=\sum_{\gamma \subset \Gamma_f} (1-t_1t_2)^{m_{\gamma}+1}(t_1t_2-1)^{\d \gamma}\\
&= \sum_{l=1}^n (1-t_1t_2)^l \sum_{\sharp S_{\gamma}=l}(-1)^{\d \gamma}\\
&= \sum_{l=1}^n (1-t_1t_2)^l \binom{n}{l}(-1)^{l-1}\\
&= 1-(t_1t_2)^n.
\end{align}
\qed

\begin{remark}
Following the proof of \cite[Theorem 5.16]{M-T-3}, we can easily give another proof to the Steenbrink conjecture which was proved  by Varchenko-Khovanskii \cite{K-V} and Saito \cite{Saito-3} independently. For an introduction to this conjecture, see an excellent survey in Kulikov \cite{Kulikov} etc.
\end{remark}

\begin{remark}
For a polynomial map $f \colon \CC^n \longrightarrow \CC$, it is well-known that there exists a finite subset $B \subset \CC$ such that the restriction
\begin{equation}
\CC^n \setminus f^{-1}(B) \longrightarrow \CC \setminus B
\end{equation}
of $f$ is a locally trivial fibration. We denote by $B_f$ the smallest such subset $B \subset \CC$. For a point $b\in B_f$, take a small circle $C_{\e}(b)=\{x\in \CC\ |\ |x-b|=\e\}$ ($0<\e\ll 1$) around $b$ such that $B_f\cap \{x \in \CC\ |\ |x-b|\leq\e\}=\{b\}$. Then by the restriction of $\CC^n \setminus f^{-1}(B_f) \longrightarrow \CC \setminus B_f$ to $C_{\e}(b)\subset \CC\setminus B_f$ we obtain a geometric monodromy automorphism $\Phi_f^b \colon f^{-1}(b+\e) \simto f^{-1}(b+\e)$ and the linear maps
\begin{equation}
\Phi_j^b \colon H^j(f^{-1}(b+\e ) ;\CC) \overset{\sim}{\longrightarrow} H^j(f^{-1}(b+\e ) ;\CC) \ \ (j=0,1,\ldots)
\end{equation}
associated to it. The eigenvalues of $\Phi_j^b$ were studied in \cite[Sections 3 and 4]{M-T-2} etc. If $f$ is tame at infinity, as in \cite[Section 4]{M-T-3} we can introduce a motivic Milnor fiber $\SS_f^{b} \in \M_{\CC}^{\hat{\mu}}$ along the central fiber $f^{-1}(b)$ to calculate the numbers of the Jordan blocks for the eigenvalues $\lambda \neq 1$ in $\Phi_{n-1}^b$. This result can be easily obtained  by using the proof of Sabbah \cite[Theorem 13.1]{Sabbah-2}. It would be an interesting problem to construct a motivic object to calculate the eigenvalue $1$ part of $\Phi_{n-1}^b$.
\end{remark}

\end{document}